\documentclass[10pt]{amsart}

\usepackage{amsmath,amssymb,amscd}\usepackage{amsthm,latexsym,graphics,xcolor}
\usepackage{enumerate,engord}
\usepackage[all]{xy}
\usepackage{pstricks-add}
\usepackage{hyperref}

\theoremstyle{plain}
    \newtheorem{thm}{Theorem}[section]
    \newtheorem{lem}[thm]   {Lemma}
    \newtheorem{cor}[thm]   {Corollary}
    \newtheorem{prop}[thm]  {Proposition}

\theoremstyle{definition}
    \newtheorem{defn}[thm]  {Definition}

    \newtheorem{rem}[thm]{Remark}

\def\cat{{\rm{cat}\hskip1pt}}
\def\zcl{{\rm{zcl}\hskip1pt}}

\def\TC{{\rm{TC}\hskip1pt}}
\def\cd{{C}}

\title[Hopf invariants, TC and LS-cat]{Hopf invariants, topological complexity, and LS-category of the cofiber of the diagonal map for two-cell complexes}

\author{Jes\'us Gonz\'alez\textsuperscript{\dag}}
\author{Mark Grant}
\author{Lucile Vandembroucq\textsuperscript{\ddag}}
\thanks{\textsuperscript{\dag}~~Partially supported by Conacyt Research Grant 221221.}
\thanks{\textsuperscript{\ddag}~~Partially supported by the Research Centre of Mathematics of the University of Minho with the Portuguese Funds from the ``Funda\c c\~ao para a Ci\^encia e a Tecnologia'', through the Project PEstOE/MAT/UI0013/2014.}

\address{Departamento de Matem\'aticas, Centro de Investigaci\'on y de Estudios Avanzados del IPN, Av.~IPN 2508, Zacatenco, M\'exico City 07000, M\'exico}
\email{jesus@math.cinvestav.mx}

\address{Department of Mathematical Sciences, University of Aberdeen, Fraser Noble Building, Meston Walk, Aberdeen AB24 3UE, UK}
\email{mark.grant@abdn.ac.uk}

\address{Universidade do Minho, Centro de Matem\'atica, Campus de Gualtar, 4710-057 Braga, Portugal}
\email{lucile@math.uminho.pt}
\begin{document}

\begin{abstract}
Let $X$ be a two-cell complex with attaching map $\alpha\colon S^q\to S^p$, and let $C_X$ be the cofiber of the diagonal inclusion $X\to X\times X$. It is shown that the topological complexity ($\TC$) of $X$ agrees with the Lusternik-Schnirelmann category ($\cat$) of $C_X$ in the (almost stable) range $q\leq2p-1$. In addition, the equality $\TC(X)=\cat(C_X)$ is proved in the (strict) metastable range $2p-1<q\leq3(p-1)$ under fairly mild conditions by making use of the Hopf invariant techniques recently developed by the authors in their study of the sectional category of arbitrary maps.
\end{abstract}

\maketitle



\bigskip\medskip\begin{center}
Draft version - \today
\end{center}

\section{Introduction}
The use of generalized Hopf invariants as obstructions for the increment of the Lusternik-Schnirelmann of a space $X$ upon cell attachments, begun by Berstein and Hilton's pioneering work~\cite{MR0126276}, played a central role in Iwase's disproof of the Ganea conjecture~\cite{MR1642747,MR1905835}. The authors of this paper have recently developed and applied in~\cite{GGVhopf} the Hopf invariant ideas to study, more generally, the sectional category of arbitrary maps. In particular, this led to an extension of Iwase's disproof of the Ganea conjecture, now in the realm of the topological complexity $\TC$, a concept introduced by Farber in~\cite{Far} to study, from a purely topological perspective, the motion planning problem in robotics.

\medskip
In this paper we apply further the Hopf invariant methods to the robotics problem. We show that the topological complexity of a two-cell complex $X$ in the metastable range agrees with $\cat(C_X)$, the LS category of the cone $C_X$ of the diagonal inclusion $\Delta\colon X\to X\times X$. Much of the motivation for such  result starts with Farber's observation in~\cite[Lemma~18.3]{MR2276952} that the inequality $\cat(C_X)\leq\TC(X)+1$ holds for any space $X$. The stronger inequality $\cat(C_X)\leq\TC(X)$ is proved in~\cite[Theorem~10]{GCVcofibre} for an $(s-1)$-connected finite cell complex $X$ ($s>0$) satisfying the reasonably mild condition
\begin{equation}\label{mild}
\dim(X)< s(\TC(X)+1)-1.
\end{equation}
More interesting is to note that the opposite inequality, $\TC(X)\leq\cat(C_X)$, is proved in~\cite[Corollary~9]{GCVcofibre} under the somehow more restrictive condition
\begin{equation}\label{restrictive}
2\dim(X)<s(\cat(C_X)+2)-1.
\end{equation}
For instance, if $X_\alpha$ stands for the cone of a map $\alpha\colon S^q\to S^p$ with trivial Bertein-Hilton-Hopf invariant $H(\alpha)$ (so that $\TC(X_{\alpha})=2$), then condition~(\ref{mild}) amounts to requiring the metastable-range condition $q\leq3(p-1)$, while~(\ref{restrictive}) amounts to the slightly stronger restriction $q\leq\frac{5}{2}p-2$ ---or to the much stronger stable-range restriction $q\leq2(p-1)$ if in fact $\cat(C_{X_\alpha})=2$. The main goal of this paper is to show that the equality $\TC(X_\alpha)=\cat(X_\alpha)$ holds in many cases of the metastable range, independently of the (non)vanishing of $H(\alpha)$:

\begin{thm}\label{t1} Let $X_\alpha$ be the cone of $\alpha: S^q\to S^p$. If $q=p=1$ or $p\leq q\leq3(p-1)$, then the equality $\TC(X_\alpha)=\cat(C_{X_\alpha})$ holds except, perhaps, when $p$ is even and $H(\alpha)$ has order 3. In the latter case (which can only hold with $2p-1<q$) we have $$2\leq\cat(C_{X_\alpha})\leq\TC({X_\alpha})\leq3.$$
\end{thm}

The relevance of Theorem~\ref{t1} stems, on the one hand, from the fact that the equality $\TC(X)=\cat(C_X)$ is known to hold for many interesting families of cell complexes $X$: closed orientable surfaces, path-connected (non-necessarily associative) $H$-spaces, closed simply connected symplectic manifolds, ordered configuration spaces of points in a Euclidean space, as well as real projective spaces~(\cite{GCVcofibre}). However, the case of a closed non-orientable surface $N_g$ is very appealing since, according to~\cite{dran}, $\TC(N_g)=4$ and $\cat(C_{N_g})=3$ provided the genus $g$ is at least 5. It would be nice to recast such a property in terms of the relevant Hopf invariants and, even more interestingly, to address the missing low-genus cases.

\medskip
By looking at tables of homotopy groups, we find that, for $2p-1<q<3(p-1)$ with $p$ even, the first group $\pi_q(S^{2p-1})$ with 3 torsion holds with $(q,p)=(14,6)$. This gives the first instance of potential maps $\alpha\colon S^q\to S^p$ in the range $q\leq3(p-1)$ for which Theorem~\ref{t1} could fail to assure the equality $\TC({X_\alpha})=\cat(C_{X_\alpha})$, depending on whether there exists such a map $\alpha$ with Hopf invariant of order three.

\section{Spheres: the typical example}
For a topological space $X$, let $C_X$ denote the cofiber of the diagonal inclusion $\Delta_X\colon X\to X\times X$. The standard fibrational substitute of $\Delta_X$ is the end-points evaluation map $e_{0,1}\colon P(X)\to X\times X$ which takes a (free) path $\gamma\colon [0,1]\to X$ to $e_{0,1}(\gamma)=(\gamma(0),\gamma(1))$. The topological complexity of $X$, denoted by $\TC(X)$, is the sectional category of $e_{0,1}$. Likewise, the Lusternik-Schnirelmann category of a based space $(X,\star)$, $\cat(X)$, is the sectional category of the evaluation map $e_1\colon P_0(X)\to X$ which takes a based path $\gamma$ on $X$ (i.e.~a path $\gamma \colon [0,1]\to X$ satisfying $\gamma(0)=\star$) to $e_1(\gamma)=\gamma(1)$.

\medskip
It is convenient to approach $\cat(X)$ through the associated Ganea fibrations $F_n(X)\to G_n(X)\to X$ with fiber inclusion and projection $i_n$ and  $g_n$, respectively. A model for these fibrations is given by the iterated $(n+1)$-fold fiberwise power of $e_1$. Likewise, the $\TC$-Ganea fibrations $F_n(X)\to G_n^{\TC}(X)\to X\times X$, with fiber inclusion and projection $i_n^{\TC}$ and $g_n^{\TC}$, can be constructed as the iterated $(n+1)$-fold fiberwise power of $e_{0,1}$. The key point is that, when $X$ is a path-connected paracompact space, the condition $\cat(X)\leq n$ is equivalent to the existence of a (pointed) global section for $g_n$. Likewise, $\TC(X)\leq n$ if and only if $g_n^{\TC}$ admits such a section. For details on these constructions and their properties, we refer the reader to~\cite{GGVhopf}, a paper which the reader will be assumed to be familiar with.

\medskip
The topological complexity of spheres,
\begin{equation}\label{FarberTCspheres}
\TC(S^n)=\begin{cases}1, & \mbox{if $n$ is odd;}\\
2, & \mbox{if $n$ is even,}
\end{cases}
\end{equation}
was computed in Farber's early $\TC$-work. The similarity between~(\ref{FarberTCspheres}) and the description of $\cat(C_X)$ in Lemma~\ref{exts} below has already been noted in~\cite{GCVcofibre}. We include a proof since this will introduce notation needed in later parts of the paper.

\begin{lem}\label{exts}
The category of the cofiber of the diagonal for spheres is given by
$$\cat(C_{S^n})=\begin{cases}1, & \mbox{if $n$ is odd;}\\
2, & \mbox{if $n$ is even.}
\end{cases}$$

\end{lem}
\begin{proof}
We start by recalling from~\cite[Proposition~28]{GLweak} the structure of $\cd_{S^n}$ as a two-cell complex. Consider the diagram
\begin{equation}\label{LJM}\xymatrix{
& S^n \ar[r] \ar[d]^{\nu}& \ast \ar[d]\\
S^{2n-1} \ar[d] \ar[r]^{[\iota_1, \iota_2]\;\;\;} &
S^n\vee S^n \ar[r]^{\;\;\;(1,-1)}
\ar@{^{(}->}[d]^{j} & S^n\ar[d]\\
\ast \ar[r]  & S^n\times S^n \ar[r]_q  & P}
\end{equation}
where $\nu$ is the comultiplication, and $(1,-1)$ stands for the map with the indicated cocomponents, so that the right top square is a homotopy pushout. Likewise, the left square is a homotopy pushout, and the right bottom square is taken to be a homotopy pushout. Since the two composed rectangles are homotopy pushouts and since $j\circ\nu\simeq\Delta$, we get $$P\simeq C_{S^n}\simeq S^n\cup_{[\iota,-\iota]}e^{2n},$$ where $\iota$ is the identity on $S^n$, and $$[\iota,-\iota]:S^{2n-1}\rightarrow S^n$$ stands for the Whitehead product. Since LS-category increases at most by one upon a cell attachment,~\cite[Theorem~3.19]{MR0126276} yields
$$
\cat(C_{S^n})=\begin{cases}
1, & \mbox{if the classical Hopf invariant of $[\iota,-\iota]$ vanishes;}\\
2, & \mbox{otherwise.}
\end{cases}
$$
The result then follows from the well known fact (see for instance~\cite[pp.~225 and~428]{AT}) that the classical Hopf invariant of $[\iota,-\iota]$ vanishes if and only if $n$ is odd.
\end{proof}

\begin{rem}
 The Hopf invariant of $[\iota,-\iota]$ is known to be $\pm2$ for $n$ even. This fact should be compared with Remark~\ref{elvalorexplicito} below.
 \end{rem}

\begin{rem}\label{cofibresuspension}
It is easy to check that, for any suspension $X=\Sigma A$, the analogue of the top right square in (\ref{LJM}) is a homotopy push-out. As a consequence, the proof of Lemma \ref{exts} generalizes to any suspension $X=\Sigma A$ giving that $C_{\Sigma A}=\Sigma A\cup_{[\Sigma A,-\Sigma A]}C(A\ast A)$, where $\Sigma A$ also stands for the identity of $\Sigma A$ and $[-,-]$ is the generalized Whitehead product.  In particular, $\cat(C_{\Sigma A})\leq 2$.
\end{rem}

Rather than the computational argument above, what we need for the purposes of the paper is the purely homotopy \emph{reason} below for the equality $\TC(S^n)=\cat(C_{S^n})$. For the generalized reason will then be applied in the next section to prove the equality $\TC(X)=\cat(C_X)$ for suitable two-cell complexes $X$ with attaching map $S^q\to S^p$ in the metastable range $q\leq3(p-1)$. The point is that the argument below for a sphere already contains all the key points featured in the situation for the metastable two-cell complex $X$. At the same time, the situation for a sphere is much more transparent than the situation for a two-cell complex, so the discussion in this section is intended to clarify the global (more technical) argument in the next section.

\medskip
In order to simplify the discussion, we assume $n\geq2$ in the following considerations. The starting point is the observation that Lemma~\ref{exts} can be proved in terms of the commutative diagram
\begin{equation}\label{gcat}\xymatrix{
&& F_1(C_{S^n}) \ar[d]^{i_1} \\
&& G_1(C_{S^n}) \ar[d]^{g_1} \\
S^{2n-1} \ar[r]^{\;[\iota,-\iota]} \ar@/^/[urur]^{h_{[\iota,-\iota]}} &
S^n \; \ar@/^/[ur]^{\sigma} \ar@{^{(}->}[r]&C_{S^n}
}\end{equation}
where the (pointed homotopy) lifting $\sigma$ exists since $\cat(S^n)=1$, so the restricted lifting $h_{[\iota,-\iota]}$ is the obstruction to extend $\sigma$ to a section for $g_1$. Note that $h_{[\iota,-\iota]}$ is really \emph{the} obstruction for sectioning $g_1$ because the latter map is a $(2n-1)$-equivalence, and the homotopy class of $\sigma$ is therefore unique (recall $n\geq2$). Note also that the inclusion of the bottom cell $S^n\hookrightarrow C_{S^n}$ is a $(2n-1)$-equivalence, so that the induced map $F_1(S^n)\to F_1(C_{S^n})$ is a $(3n-2)$-equivalence. Since the bottom cell of $F_1(S^n)$ splits off as a wedge summand, the homotopy class of $h_{[\iota,-\iota]}$ is fully determined by the degree of the first map in any homotopy factorization $$S^{2n-1}\to S^{2n-1}\hookrightarrow F_1(S^n)\to F_1(C_{S^n})$$ of $h_{[\iota,-\iota]}$. Of course, the degree interpretation gives the integer-represented Hopf invariant of $[\iota,-\iota]$.

\medskip
As explained in~\cite[Example~4.6]{GGVhopf}, the above argument spells out the proof of Lemma~\ref{exts} given in terms of~\cite[Theorem~3.19]{MR0126276}. In fact, much of the \emph{raison d'\^etre} of~\cite{GGVhopf} is that the method is fully generalizable and so, from this point on, we will make free use of the methods and results in~\cite{GGVhopf}, assuming the reader is familiar with that work.

\medskip
The top $\TC$-Hopf set obstructing the inequality $\TC(S^n)\leq1$ arises from the (pointed) homotopy commutative diagram
\begin{equation}\label{gtc}\xymatrix{
&& F_1( S^n) \ar[d]^{i_1^{\TC}} \\
&& G_1^{\TC}(S^n) \ar[d]^{g_1^{\TC}} \\
S^{2n-1} \ar[r]^{[\iota_1,\iota_2]\;\;\;} \ar@/^/[urur]^h &
S^n \vee S^n\; \ar@/^/[ur]^{\phi} \ar@{^{(}->}[r]&S^n\times S^n.
}\end{equation}
Here $[\iota_1,\iota_2]$ is the Whitehead product of the two inclusions $\iota_j\colon S^n\hookrightarrow S^n\vee S^n$ $(j=1,2)$, so the row is a cofiber sequence. The lifting $\phi$ exists since $S^n\vee S^n$ is a suspension, so that $$\TC_{S^n\vee S^n}(S^n)\leq\cat(S^n\vee S^n)\leq1.$$ These two inequalities are sharp in view of~\cite[Proposition 3.8(5)]{GGVhopf}: if $x\in H^n(S^n)$ is the generator, then $x\otimes1-1\otimes x$ is a zero-divisor detected on $S^n\vee S^n$. Further, since $F_1(S^n)$ is $(2n-2)$-connected, the map $g_1^{\TC}(S^n)$ is a $(2n-1)$-equivalence, so the lifting $\phi$ is unique (once again, we are using the blanket assumption $n\geq2$). Consequently, the Hopf set under consideration is the singleton consisting of the map $h$ ---the lifting to the fiber of the pointed composition $\phi\circ[\iota_1,\iota_2]$.

\medskip
Diagrams~(\ref{gcat}),~(\ref{gtc}), and the bottom right square in~(\ref{LJM}) can be combined into the larger homotopy commutative diagram
\begin{equation}\label{combed}\footnotesize
\xymatrix{
&&&&F_1(S^n)\ar[dl]^{i_1^{\TC}} \ar@{-->}[ddd]^{Q'_1}\\
&&&G_1^{\TC}(S^n)\ar[dl]^{g_1^{\TC}} \ar@{-->}[ddd]\\
S^{2n-1} \ar[r]_{[\iota_1,\iota_2]\;\;\;} \ar@/^3pc/[ururrr]^>>>>>>>>>>>>>>>h \ar@{=}[ddd] &
S^n\vee S^n\, \ar@{^{(}->}[r]
\ar@/^2pc/[urr]^>>>>>\phi \ar[ddd]_<<<<<<<<<<<<<<<<{(1,-1)} & S^n\times S^n \ar[ddd]_q\\
&&&&F_1(C_{S^n})\ar[dl]^{i_1}\\ &&&G_1(C_{S^n})\ar[dl]^{g_1}\\
S^{2n-1}\ar[r]_{[\iota,-\iota]} \ar@/^3pc/[ururrr]^>>>>>>>>>>>>>>{h_{[\iota,-\iota]\;\;\;}} & S^n\, \ar@{^{(}->}[r] \ar@/^2pc/[urr]^>>>>>{\sigma\;\;} & C_{S^n}
}\end{equation}
where the two dashed maps lie over $q$ and are obtained, by naturality of the join construction, from the commutative diagram
\begin{equation}\label{lasqs}\xymatrix{
\Omega(S^n)\ar[r]^{Q'} \ar[d] & \Omega(C_{S^n})\ar[d]\\
P(S^n)\ar[r]^{Q} \ar[d]^{e_{0,1}} & P_0(C_{S^n})\ar[d]^{e_1}\\
S^n\times S^n \ar[r]_q & C_{S^n}.
}\end{equation}
Such a diagram exists because $e_{0,1}$ is a fibrational replacement of the diagonal $\Delta\colon S^n\to S^n\times S^n$, and since the composition $q\circ\Delta$ is homotopically trivial.

\begin{lem}\label{BMconntty}
The map $Q'_1$ is a $(3n-2)$-equivalence.
\end{lem}
\begin{proof}
As in~\cite{GCVcofibre} (see Theorem~10(b) and its proof), we have $Q'\simeq\Omega(q\circ j_1)$, where $j_1:S^n \to S^n\times S^n$ is the inclusion on the first factor. On the other hand, the lower right square of~(\ref{LJM}) implies that $q\circ j_1$ is homotopic to the inclusion of the bottom cell. Thus $Q'$ is a $(2n-2)$-equivalence, and the result follows from a standard homology calculation.
\end{proof}

\begin{lem}\label{remcom}
The two combed squares in~(\ref{combed}), namely the one involving the liftings $\phi$ and $\sigma$, and the one involving the Hopf invariants $h$ and $h_{[\iota,-\iota]}$, are homotopy commutative.
\end{lem}
\begin{proof}
The square relating $\phi$ and $\sigma$ commutes because ($n<2n-1$ and) $g_1$ is a $(2n-1)$-equivalence. The commutativity of the square relating $h$ and $h_{[\iota,-\iota]}$ then follows from the well known fact that $i_1$ induces a monomorphism in each positive dimensional homotopy group.
\end{proof}

\begin{cor}
The triviality of the $\TC$-Hopf invariant $h$ is equivalent to that of the $\cat$-Hopf invariant $h_{[\iota,-\iota]}$. Consequently $\TC(S^n)=\cat(C_{S^n})$.
\end{cor}
\begin{proof}
Just note that $\TC_{S^n\vee S^n}(S^n)=1=\cat(S^n)$, where the first equality has been pointed out right after~(\ref{gtc}).
\end{proof}

\begin{rem}\label{elvalorexplicito}
Together with~\cite[Lemma~6.2]{GGVhopf}, the above considerations give an alternative proof of the well known equality $h_{[\iota,-\iota]}=\pm(1+(-1)^n)$.
\end{rem}

\section{Two-cell complexes in the metastable range with non-trivial Hopf invariant}

Throughout this section $X$ stands for a two-cell complex $S^p\cup_\alpha e^{q+1}$ whose attaching map $\alpha: S^q\to S^p$ ($q\geq p\geq2$) lies in the metastable range $2p-1< q\le 3p-3$ (so in fact $p\geq3$ and $q\geq p+3$), and has non-vanishing Hopf invariant $H(\alpha)$. As recalled after the proof of~\cite[Theorem~5.2]{GGVhopf}, $H(\alpha)$ factors (due to the metastable range hypothesis) as $$S^q\stackrel{H_0(\alpha)}{-\!\!\!-\!\!\!\longrightarrow} S^{2p-1}\stackrel{i}\hookrightarrow F_1(S^p)$$ where $i$ is the bottom cell inclusion, which splits as a wedge summand, so that $H(\alpha)$ can be identified directly with the stable map $H_0(\alpha)$. Recall in addition that, in Hilton-Whitehead's definition, $H(\alpha)$ is the obstruction for $\alpha$ to be a co-H-map. Namely, the fiber of the inclusion $S^p\vee S^p\hookrightarrow S^p\times S^p$ is $F_1(S^p)$, and the fiber inclusion restricted to the bottom cell is $[\iota_1,\iota_2]$. So, if $\nu$ stands for pinch maps, we have by definition
\begin{equation}\label{obstrusum}
\nu\circ\alpha-(\alpha\vee\alpha)\circ\nu=[\iota_1,\iota_2]\circ H_0(\alpha)\in\pi_q(S^p\vee S^p).
\end{equation}
Equivalently, since the homotopy pullback of the inclusion $j\colon S^p\vee S^p\hookrightarrow S^p\times S^p$ along the diagonal $\Delta\colon S^p\to S^p\times S^p$ is $\varepsilon\colon\Sigma\Omega S^p\to S^p$, the adjoint of the identity on $\Omega S^p$ (i.e.~the first Ganea map $G_1(S^p)\to S^p$), and since the pinch map $\nu$ (the unique homotopy lifting of $\Delta$ along $j$) corresponds to the canonical section of $\varepsilon$ (i.e.~the inclusion $\kappa\colon S^p\hookrightarrow\Sigma\Omega S^p$ of the bottom cell), we see that, by definition, the difference of the two compositions in the diagram
\begin{equation}\label{alaHilton}\xymatrix{
\Sigma\Omega S^q \ar[r]^{\Sigma\Omega\alpha\;\;}&\Sigma\Omega S^p\\
S^q\ar[r]_\alpha\ar[u]^{\kappa}&S^p\ar[u]^{\kappa}
}\end{equation}
is the image of the Hopf invariant $H(\alpha)$ under the fiber inclusion $F_1(S^p)\to \Sigma\Omega S^p$ or, equivalently, the image of $H_0(\alpha)$ under the inclusion $$j_2\colon S^{2p-1}\hookrightarrow\Sigma\Omega S^p\simeq S^p\vee S^{2p-1}\vee S^{3p-2}\cdots$$ of the next-to-the-bottom cell. What we need to record from this discussion is the well known relation in~(\ref{giladj}) below. Namely, since the map $\kappa$ on the left of~(\ref{alaHilton}) can be seen as the suspension of the bottom cell inclusion $S^{q-1}\hookrightarrow\Omega S^q$, the composite $\Sigma\Omega\alpha\circ\kappa$ is homotopic to $\Sigma\alpha'$, the suspension of the adjoint of $\alpha$. Therefore
\begin{equation}\label{giladj}
\Sigma\alpha'-\kappa\circ\alpha=j_2\circ H_0(\alpha).
\end{equation}

Consider the cone decomposition
$$\ast=C_0\subset C_1\subset C_2\subset C_3\subset C_4=X\times X$$ given by $C_1=S^p\vee S^p$, $C_2=(X\vee X)\cup(S^p\times S^p)$, and $C_3=(X\times S^p)\cup(S^p\times X)$, and obvious attaching maps
\begin{equation}\label{lasetapasC}
\Sigma_i\stackrel{\alpha_i}\longrightarrow C_i\hookrightarrow C_{i+1}
\end{equation}
where $\Sigma_0=S^{p-1}\vee S^{p-1}$, $\Sigma_1=S^{q}\vee S^{2p-1}\vee S^{q}$, $\Sigma_2=S^{p+q}\vee S^{p+q}$ and $\Sigma_3=S^{2q+1}$.

\medskip
By dimensional reasons, the diagonal map $\Delta\colon X\to X\times X$ can be deformed to a (unique up to homotopy) map $\Delta_{2}:X\to C_2$, and this yields maps $\Delta_{i}:X\to C_i$ ($i=3,4$) fitting in the homotopy commutative diagram
\begin{equation}\label{4x2}\xymatrix{
S^p\,\ar@{^(->}[r]\ar[d]^{\nu} & X \ar@{=}[r] \ar[d]^{\Delta_{2}} & X \ar@{=}[r] \ar[d]^{\Delta_{3}} &X\ar[d]^{\Delta}\\
C_ 1 \ar@{^(->}[r] & C_ 2\ar@{^(->}[r] &C_ 3 \ar@{^(->}[r] & X\times X.
}\end{equation}
(For the homotopy commutativity of the left-most square, keep in mind that $C_2\hookrightarrow X\times X$ is a $(p+q)$-equivalence.)

\begin{prop}\label{4x3}
There is an extended homotopy commutative diagram
\begin{equation}\label{sequenceofDs}\xymatrix{
S^p\,\ar@{^(->}[r]\ar[d]^{\nu} & X \ar@{=}[r] \ar[d]^{\Delta_{2}} & X \ar@{=}[r] \ar[d]^{\Delta_{3}} &X\ar[d]^{\Delta}\\
C_ 1 \ar@{^(->}[r]\ar[d]^{q_{1}} & C_ 2 \ar@{^(->}[r] \ar[d]^{q_{2}} &C_ 3 \ar@{^(->}[r]\ar[d]^{q_{3}} & X\times X\ar[d]^{q_{}}\\
D_ 1 \ar@{^(->}[r] & D_ 2 \ar@{^(->}[r] &D_ 3 \ar@{^(->}[r] & C_X
}\end{equation}
whose columns are cofiber sequences, and whose bottom row yields a cone decomposition $$\star=D_0\subset D_1\subset D_2\subset D_3\subset D_4=C_X$$ with attaching maps of the form
\begin{equation}\label{lasetapasD}
S_i \stackrel{\beta_i}{\longrightarrow}D_{i}\hookrightarrow D_{i+1}.
\end{equation}
Here $S_0=S^{p-1}$, $S_1=S^{2p-1}\vee S^q$, $S_2=S^{p+q}\vee S^{p+q}$, and $S_3=S^{2q+1}$. The cofiber sequences~(\ref{lasetapasC}) and ~(\ref{lasetapasD}) fit in homotopy commutative diagrams
\begin{equation}\label{comparaetapas}\xymatrix{
\Sigma_i\ar[r]^{\alpha_i}\ar[d]^{\tau_{i}} & C_i \ar@{^{(}->}[r] \ar[d]^{q_{i}} & C_{i+1} \ar[d]^{q_{i+1}} \\ S_i \ar[r]^{\beta_i} & D_i \ar@{^{(}->}[r] & D_{i+1}
}\end{equation}
where $\tau_2$ and $\tau_3$ are homotopy equivalences.
\end{prop}
\begin{proof}
Write the equality in~(\ref{obstrusum}) as $$\nu\circ\alpha=(\alpha\vee\alpha)\circ\nu+[\iota_1,\iota_2]\circ H_0(\alpha)$$ and note that the latter sum decomposes as $$S^q \stackrel{(1,1,H_0(\alpha))\;}{-\!\!\!-\!\!\!-\!\!\!-\!\!\!-\!\!\!-\!\!\!-\!\!\!\longrightarrow} S^q\vee S^q \vee S^{2p-1}\stackrel{(\alpha,\alpha,[\iota_1,\iota_2])\;}{-\!\!\!-\!\!\!-\!\!\!-\!\!\!-\!\!\!-\!\!\!-\!\!\!\longrightarrow} S^p\vee S^p,$$ where $(1,1,H_0(\alpha))$ is the composition
$$
S^q\stackrel{\nu}\longrightarrow S^q\vee S^q \stackrel{\nu\hspace{.3mm}\vee H_0(\alpha)\;}{-\!\!\!-\!\!\!-\!\!\!-\!\!\!-\!\!\!\longrightarrow} S^q\vee S^q \vee S^{2p-1},
$$
and $(\alpha,\alpha,[\iota_1,\iota_2])$ is the composition
$$
S^q\vee S^q \vee S^{2p-1} \stackrel{\alpha\vee\alpha\vee[\iota_1,\iota_2]\;}{-\!\!\!-\!\!\!-\!\!\!-\!\!\!-\!\!\!-\!\!\!-\!\!\!\longrightarrow} S^p\vee S^p\vee S^p\vee S^p \stackrel{\nabla}{\longrightarrow} S^p\vee S^p.
$$
We thus have the homotopy commutative diagram
$$\xymatrix{
S^q \ar[d]_-{(1,1,H_0(\alpha))} \ar[rr]^{\alpha} && S^p \ar[d]^{\nu}\\
 S^q\vee S^q \vee S^{2p-1} \ar[rr]_{\hspace{5.5mm}(\alpha,\alpha,[\iota_1,\iota_2])} && S^p\vee S^p,}$$
which is then extended to the 3-by-3 homotopy commutative diagram below by taking cofibers of rows and columns.
\begin{equation}\label{lucilesolution}\xymatrix{
S^q \ar[d]_{(1,1,H_0(\alpha))} \ar[rr]^{\alpha} && S^p\, \ar[d]^{\nu} \ar@{^{(}->}[r] & X \ar[d]\\  S^q\vee S^q \vee S^{2p-1} \ar[rr]_{\hspace{5mm}(\alpha, \alpha,[\iota_1,\iota_2])} \ar[d]&& S^p\vee S^p\, \ar@{^{(}->}[r] \ar[d] & C_2 \ar[d] \\  Y \ar[rr] && D_1 \ar@{^{(}->}[r] & D_2
}\end{equation}
In view of the left-most square in~(\ref{4x2}), the top right-most vertical map $X\to C_ 2$ can be chosen to be $\Delta_{2}$. Note also that the top right square in~(\ref{LJM}) shows that $D_1$ has the homotopy type of $S^p$. Further, the map $(1,1,H_0(\alpha))$ lies in the stable range, so that its cofibre $Y$ is a 1-connected suspension. In fact, the left-most vertical cofiber sequence in~(\ref{lucilesolution}) induces short exact sequences in integral homology, from which it is easy to see that $Y$ has the homology type and, then, the homotopy type of $S^{2p-1}\vee S^q$. In particular, $D_2$ has the homotopy type of a three-cell complex $S^p\cup e^{2p}\cup e^{q+1}$. This yields the assertions relevant for the first two columns in~(\ref{sequenceofDs}). For instance, the map $\tau_0$ in~(\ref{comparaetapas}) corresponds to $(1,-1)\colon S^{p-1}\vee S^{p-1}\to S^{p-1}$, while $\tau_1$ is the left bottom vertical map in~(\ref{lucilesolution}).
The rest of the assertions follow easily by extending each of the commutative squares
$$\xymatrix{
\star \ar[r] \ar[d] & X\ar[d]^{\Delta_{2}} && \star \ar[r] \ar[d] & X\ar[d]^{\Delta_{3}} \\ S^{p+q}\vee S^{p+q}\ar[r]_{\hspace{6mm}\alpha_2} & C_2 && S^{2q+1} \ar[r]_{\alpha_3} & C_3
}$$
to corresponding 3-by-3 homotopy commutative diagrams of cofibrations analogous to~(\ref{lucilesolution}).
\end{proof}

Next we compare the $\cat$-Hopf sets arising from the cofiber sequences~(\ref{lasetapasD}) with the $\TC$-Hopf sets arising from the cofiber sequences~(\ref{lasetapasC}) ---the latter studied in~\cite{GGVhopf}. Except for a few additional technical considerations, the method will be the one used in the previous section for explaining, from a homotopical viewpoint, the equality $\TC(S^n)=\cat(C_{S^n})$.

\medskip
First we need the analogues of~(\ref{combed}) and~(\ref{lasqs}). For the latter, consider the commutative diagrams
$$\xymatrix{
\Omega(X)\ar[r]^{Q'} \ar[d] & \Omega(C_{X})\ar[d] && F_n(X)\ar[r]^{Q'_n} \ar[d] & F_n(C_X)\ar[d]\\
P(X)\ar[r]^{Q} \ar[d]^{e_{0,1}} & P_0(C_{X})\ar[d]^{e_1} && G_n^{\TC}(X)\ar[r]^{Q_n} \ar[d]_{g_n^{\TC}} & G_n(C_X)\ar[d]^{g_n}\\
X\times X \ar[r]_q & C_{X} && X\times X \ar[r]_q & C_X,
}$$
where the second one is obtained from the first one in terms of the fiberwise join construction.

As in the case of spheres, the arguments in the proof of~\cite[Theorem~10(b)]{GCVcofibre} show that $Q'$ is homotopic to $\Omega(q\circ j_1)$, where $j_1:X \to X\times X$ is the inclusion on the first factor. Since $q\circ j_1$ is a $(2p-1)$-equivalence, we get:

\begin{lem}\label{equirango}
$Q'_n$ is a $(p(n+2)-2)$-equivalence.
\end{lem}

Let $\gamma_n\colon \Gamma_n\to X\times X$ denote the pullback of $g_n$ along $q$, so that the induced map $\mathcal{Q}_n\colon G_n^{\TC}(X)\to \Gamma_n$ is a $(p(n+2)-2)$-equivalence. These maps fit into the commutative diagram
$$\xymatrix{
F_n(X) \ar[rr]^{Q'_n} \ar[d] && F_n(C_X)\ar[d]\ar[ld] \\
G_n^{\TC}(X)\ar[r]^{\hspace{3mm}\mathcal{Q}_n} \ar[d]_{g_n^{\TC}} & \Gamma_n \ar[r]_{q_n\hspace{4.5mm}} \ar[dl]^{\gamma_n}& G_n(C_X)\ar[d]^{g_n}\\
X\times X \ar[rr]_q && C_X
}$$
and, by restriction (with respect to the bottom squares in~(\ref{sequenceofDs})), we get, for $1\leq i\leq4$, the top part (without the dotted maps) of the following commutative 3D-diagram:
\begin{equation}\label{masdia}\xymatrix{
&&&F_n(X)\ar[rr]^{Q'_n} \ar[d] && F_n(C_X)\ar[d]\ar[ld]\\
&&&G_{n,i}^{\TC}(X)\ar[r]^{\hspace{2.6mm}\mathcal{Q}_{n,i}} \ar[d]|-<<<{\rule{0mm}{2.5mm}\mbox{\tiny$g_{n,i}^{\TC}$}} &\Gamma_{n,i} \ar[r]^{q_{n,i}\hspace{4.5mm}} \ar[dl]^{\gamma_{n,i}} & G_{n,i}(C_X)\ar[d]|-<<<<{\rule{0mm}{2mm}\raisebox{1mm}{\scriptsize$g_{n,i}$}}\\
&&&C_{i} \ar[rr]_{\hspace{9mm}q_{i}} && D_{i} \\
&&C_{i-1} \ar@{.>}[ruu]|->>>>>>>{\rule{0mm}{4.5mm}\raisebox{2.25mm}{${}_{\lambda_{i-1}^{\TC}}$}} \ar[ru]\ar[rr]_{q_{i-1}} && D_{i-1} \ar[ru] \ar@{.>}[ruu]|->>>>>>>{\rule{0mm}{4.5mm}\raisebox{2.25mm}{${}_{\lambda_{i-1}^{\cat}}$}} \\
\Sigma_{i-1} \ar@{.>}[rrruuuu]|->>>>>>>>>>>>>>>>>>>>>>>>{\rule{0mm}{4.5mm}\raisebox{2.25mm}{${}_{h_{i-1}^{\TC}}$}} \ar[rrr]^>>>>>>>>>>>>>{\tau_{i-1}} \ar[rru]_{\alpha_{i-1}} &&& S_{i-1} \ar[ru]_{\beta_{i-1}}
\ar@{.>}[rruuuu]|->>>>>>>>>>>>>>>>>>>>>>{\rule{0mm}{4.5mm}\raisebox{2.25mm}{${}_{h_{i-1}^{\cat}}$}}
}\end{equation}
Note that $\mathcal{Q}_{n,i}$ inherits the connectivity properties of~$\mathcal{Q}_n$ and $Q'_n$.

\begin{prop}\label{puntodeinicio}
The bottom two Hopf sets coming from each of the ``walls'' in~(\ref{masdia}) are non-trivial, that is,
$\TC_{C_2}(X)=\cat_{D_2}(C_X)=2$.
\end{prop}
\begin{proof}
The fact that $\TC_{C_2}(X)=2$ is proved in~\cite[Example~5.3]{GGVhopf}, whereas the inequality $\cat_{D_2}(C_X)\leq2$ holds by cone-length considerations~(\cite[Proposition~3.9]{GGVhopf}). To complete the proof, assume for a contradiction that $\cat_{D_2}(C_X)\leq1$. Then the map $g_{1,2}$ in~(\ref{masdia}) admits a section, which can then be pulled back to a section of $\gamma_{1,2}$. The latter section factors (up to homotopy) through the $(3p-2)$-equivalence $\mathcal{Q}_{1,2}$, since $\dim(C_2)=q+1\leq3p-2$. This yields a section of $g_{1,2}^{\TC}$, which contradicts the fact that $\TC_{C_2}(X)=2$.
\end{proof}

\begin{defn}\label{comptbs}
In the setting of~(\ref{masdia}), liftings $\lambda^{\TC}_{i-1}$ and $\lambda^{\cat}_{i-1}$ of, respectively, $g_{n,i}^\TC$ and $g_{n,i}$ are said to be \emph{compatible} provided $\lambda^{\cat}_{i-1}\circ q_{i-1}\simeq q_{n,i}\circ \mathcal{Q}_{n,i}\circ\lambda^{\TC}_{i-1}$. Note that, in such a case, the resulting maps $h_{i-1}^{\TC}$ and $h_{i-1}^{\cat}$ are compatible in the sense that $h_{i-1}^{\cat}\circ\tau_{i-1}\simeq Q'_n\circ h_{i-1}^{\TC}$.
\end{defn}

\begin{lem}\label{levantamientoscompatibles}
Assume $n\geq2$ and $i\geq3$ in~(\ref{masdia}). For any lifting $\lambda^{\cat}_{i-1}$ of $g_{n,i}$ there is a compatible lifting $\lambda^{\TC}_{i-1}$ of $g_{n,i}^{\TC}$. Conversely, for any lifting $\lambda^{\TC}_{i-1}$ of $g_{n,i}^{\TC}$ there is a compatible lifting~$\lambda^{\cat}_{i-1}$ of $g_{n,i}$.
\end{lem}
\begin{proof}
Since $\dim(C_2)\leq\dim(C_3)=p+q+1\leq4p-2$, and since $\mathcal{Q}_{n,i}$ is a $(4p-2)$-equivalence, the argument in the previous proof applies to prove the first assertion. For the converse, note first that a lifting $\lambda_{i-1}^{\TC}$ corresponds to a section of $g_{n,i-1}^{\TC}$. Likewise, a lifting $\lambda_{i-1}^{\cat}$ corresponds to a section of $g_{n,i-1}$. Furthermore, the compatibility of the sections implies the compatibility of the liftings. The result then follows from~\cite[Lemma 4.2]{GGVhopf} using the cofibre sequence $X\to C_{i-1}\to D_{i-1}$. Namely, a section of $g_{n,i-1}^{\TC}$ yields a compatible section of $g_{n,i-1}$, since $\dim(\Sigma X)=q+2$ and since $g_{n,i-1}$ is (at least) a $(3p-1)$-equivalence ---the latter fact is due to the obvious connectivity of $F_n(C_X)$.
\end{proof}

The following consequence should be compared to the considerations following~(\ref{mild}) in the introduction:
\begin{cor}
$\cat(C_X)\leq\TC(X)$, with equality if $\TC(X)=2$.
\end{cor}
\begin{proof}
Since $\tau_2$ and $\tau_3$ are homotopy equivalences, Lemma~\ref{levantamientoscompatibles} implies that the triviality of any of the two top Hopf sets on the ``right wall'' of~(\ref{masdia}) follows from the triviality of the corresponding Hopf set on the ``left wall''.
\end{proof}

Proposition~\ref{sighopset} below, which is a partial refinement of Lemma~\ref{levantamientoscompatibles}, follows directly from~\cite[Proposition~4.5]{GGVhopf}:
\begin{prop}\label{sighopset}
The Hopf sets associated to both walls in~(\ref{masdia}) are singletons provided $(n,i)=(2,3)$ or $(n,i)=(3,4)$.
\end{prop}

\begin{cor}\label{hopfnivel3}
$$\cat_{D_3}(C_X)=\TC_{C_3}(X)=\begin{cases}
3, & \mbox{if \ $(2+(-1)^p)H(\alpha)\neq0;$}\\
2, & \mbox{otherwise.}
\end{cases}
$$
\end{cor}
\begin{proof}
The (single-valued) Hopf sets in~(\ref{masdia}) for $(n,i)=(2,3)$ lie in (a sum of) $(p+q)$-dimensional homotopy groups. Further, the resulting $\TC$-Hopf \emph{invariant} is mapped into the $\cat$-Hopf \emph{invariant} by the map induced by $Q'_2$, which is a $(4p-2)$-equivalence. This yields the first equality; the second one is a direct consequence of~\cite[Theorem~5.5]{GGVhopf}.
\end{proof}

We are only one lemma away from giving the proof of Theorem~\ref{t1} under the conditions in force in this section, namely when the attaching map $\alpha\colon S^q\to S^p$ lies in the metastable range $2p-1<q\leq3(p-1)$, and has non-trivial Berstein-Hilton-Hopf invariant $H(\alpha)$. The most interesting case holds with $(2+(-1)^p)H(\alpha)\neq 0$, for then Proposition~\ref{sighopset} and Corollary~\ref{hopfnivel3} imply that the relevant (top) $\TC$- and $\cat$-Hopf sets are described, with trivial indeterminacy, by~(\ref{masdia}) with $(n,i)=(3,4)$. Note that, in such a case, the argument in the proof of Corollary~\ref{hopfnivel3} proves Theorem~\ref{t1} if the metastable hypothesis $q\leq3(p-1)$ is replaced by the stronger condition $q\leq\frac{5}{2}p-2$. Lemma~\ref{esqueleto} below allows us to maneuver using only the less restrictive hypothesis.

\begin{lem}\label{esqueleto}
There is a CW structure on $F_3(C_{S^p})$ with $(6p-4)$-skeleton given by
\begin{equation}\label{lusp}
S^{4p-1}\vee \bigvee\limits_{4}\left(S^{5p-2}\cup_h e^{5p-1}\right)
\end{equation}
where $h$ is the classical (integer-represented) Hopf invariant of the Whitehead product $[\iota,-\iota]$, and $\iota$ is the identity on $S^p$. Further, if $F_3(C_{S^ p})\to F_3(C_X)$ is the map induced by the  inclusion of the bottom cell $S^p\hookrightarrow X$, then the composition $$S^{4p-1}\vee \bigvee\limits_{4}(S^{5p-2}\cup_h e^{5p-1}) \hookrightarrow F_3(C_{S^ p})\to F_3(C_X)$$ is a $(3p+q-1)$-equivalence.
\end{lem}
\begin{proof}
The bottom cell inclusion $S^p\hookrightarrow X$ induces a $q$-equivalence $C_{S^ p}\to C_X$ (because both $S^ p \hookrightarrow X$ and $S^ p\times S^ p\hookrightarrow X\times X$ are $q$-equivalences) which, as in the considerations following~(\ref{llh3}), yields a $(3p+q-1)$-equivalence $F_3(C_{S^ p})\to  F_3(C_X)$. Thus, it remains to show the first assertion of the lemma.

\smallskip
Let $\iota$ be the identity on $S^p$. Applying~\cite[Proposition 4.3]{Gilbert} to the cofiber sequence
$$\xymatrix{
S^{2p-1} \ar[r]^{\;\;[\iota,-\iota]} & S^p \ar[r] & C_{S^p}
}$$
(note that Gilbert's hypothesis that $S^p$ be 2-connected holds in the metastable range $2p-1<q\leq3(p-1)$), we get a $(3p-2)$-equivalence
$$\mathrm{Cone}(\Sigma [\iota,-\iota]') \to \Sigma\Omega C_{S^ p}$$
where $[\iota,-\iota]'$ is the adjoint of $[\iota,-\iota]$. Suspending once and using~(\ref{giladj}), we get a $(3p-1)$-equivalence
$$\rho\colon \mathrm{Cone}(\Sigma(j_2\circ H_0([\iota,-\iota]))) \to \Sigma^2\Omega C_{S^ p},$$
because a Whitehead product has trivial suspension. Note that the domain of $\rho$ is
$$S^{p+1}\vee \left(S^ {2p}\cup_he^ {2p+1}\right)\vee S^{3p-1}\vee S^{4p-2}\vee S^{5p-3}\vee\cdots.$$
In particular, the restriction of $\rho$ to the $(3p-2)$-skeleton of its domain is a $(3p-2)$-equivalence
$$
\rho_1\colon \Sigma^2 L\to\Sigma^2\Omega C_{S^p},
$$
where $L=S^{p-1}\vee M^{2p-1}_h$, and $M^{2p-1}_h$ stands for the $h$-torsion Moore space of dimension $2p-1$. The desired conclusion is now a standard exercise in homotopy theory, and we just sketch the details. Homology calculations show that both $1_L\wedge\rho_1$ and $\rho_1\wedge1_{\Omega C_{S^p}}$ are $(4p-3)$-equivalences, which yields a $(4p-3)$-equivalence $$\rho_2\colon\Sigma^2 L^{\wedge2}\to\Sigma^2(\Omega C_{S^p})^{\wedge2}.$$ The process repeats two more times to yield a $(6p-5)$-equivalence
$$\rho_4\colon\Sigma^2 L^{\wedge4}\to\Sigma^2(\Omega C_{S^p})^{\wedge4}.$$ The conclusion follows by observing that the $(6p-4)$-skeleton of the domain of the $(6p-4)$-equivalence $\Sigma\rho_4\colon\Sigma^3L^{\wedge4}\to F_3(C_{S^p})$ is the space described in~(\ref{lusp}).
\end{proof}

\begin{proof}[Proof of Theorem~\ref{t1}] (Assuming $(2+(-1)^p)H(\alpha)\neq 0$, $2p-1<q\leq3(p-1)$, and $2\leq p$.)
We have noted that the Hopf sets under consideration are single valued, and correspond to the compatible maps $h_3^{\TC}$ and $h_3^{\cat}$ in~(\ref{masdia}) with $(n,i)=(3,4)$.

\smallskip
The homotopy class $h_3^{\TC}$ is well understood in terms of $H_0(\alpha)\circledast H_0(\alpha)$, the join-square of the Hopf invariant $H_0(\alpha)$: As explained at the beginning of the section, $H(\alpha)$ can be thought of as a map $H_0(\alpha)\colon S^q\to S^{2p-1}$ and, in these terms, $h_3^{\TC}$ is the composition
\begin{equation}\label{llh3}\xymatrix{
S^{2q+1} \ar[rrrr]^{2(2+(-1)^p)\cdot H_0(\alpha)\circledast H_0(\alpha)} &&&& S^{4p-1} \ar@{^{(}->}[r] & F_3(S^p)\ar[r] & F_3(X),
}\end{equation}
where the middle map is the inclusion of the bottom cell in $F_3(S^p)$, and the map on the right of~(\ref{llh3}) is induced by the inclusion of the bottom cell in $X$ ---c.f.~\cite[Corollary~4.13, Theorem~5.4, and their proofs]{GGVhopf}. Note that the composition of the last two maps in~(\ref{llh3}) yields the inclusion of the bottom cell in $F_3(X)$. In fact, since the map on the right of~(\ref{llh3}) is a $(3p+q-1)$-equivalence, and since the bottom cell in $F_3(S^p)$ is well known to split off as a wedge summand, $h_3^{\TC}$ can simply be thought of as being given by the first map in~(\ref{llh3}).

\smallskip
Now recall from Lemma~\ref{equirango} that $Q'_3$ is a $(5p-2)$-equivalence, so it has degree~$\pm1$ on the bottom cell. Since $\tau_3$ is a homotopy equivalence, we see from~(\ref{masdia}) that $h_3^{\cat}$ is given up to a sign by the composition
$$\xymatrix{
S^{2q+1} \ar[rrrr]^{2(2+(-1)^p)\cdot H_0(\alpha)\circledast H_0(\alpha)} &&&& S^{4p-1} \ar@{^{(}->}[r] & F_3(C_X)
}$$
where, once again, the latter map is inclusion of the bottom cell. The result follows since Lemma~\ref{esqueleto} allows us to identify $h_3^{\cat}$ with the first map in~(\ref{llh3}).
\end{proof}

Our methods also yield:

\begin{cor}
The following conditions are equivalent:
\begin{itemize}
\item $\TC(X)=4$.
\item $\cat(C_X)=4$.
\item $2(2+(-1)^p)\cdot H_0(\alpha)\circledast H_0(\alpha)\neq0$.
\end{itemize}
\end{cor}

\begin{proof}[Proof of Theorem~\ref{t1}] (Assuming $H(\alpha){\,\neq\,}0$, $2p{\,-\,}1{\,<\,}q\leq3(p-1)$, and $2\leq p$.) For the first assertion of the theorem, we can assume that $p$ is odd or that $3H(\alpha)\neq0$. In either case, the non-vanishing of $H(\alpha)$ implies the non-vanishing of $(2+(-1)^p)H(\alpha)$, and the previous proof applies.

\smallskip
For the second assertion of the theorem, just note that the Hopf-set approach also shows that the only instance where the equality $\TC(X)=\cat(C_X)$ can fail (having actually $\TC(X)=\cat(C_X)+1$) would hold with $\TC(X)=3$ due to a vanishing third $\TC$-Hopf invariant $(2+(-1)^p)H_0(\alpha)$, followed by a non-trivial fourth Hopf set (in dimension 2).
\end{proof}

\section{Non-Hopf-sets methods}

As a consequence of \cite[Theorem 24]{GLweak}, we have:
\begin{lem}\label{markweigths}
The zero-divisors cup-length of $X$ (with any ring of coefficients) is a lower bound for $\cat(C_X)$.
\end{lem}
For the convenience of a forthcoming proof we give here a direct proof of this lemma:
\begin{proof}
The projection onto the first axis $X\times X\to X$ is a retraction for the diagonal $\Delta\colon X\to X\times X$. Thus the exact cohomology sequence of the pair $(X \times X, X)$ splits, and the reduced cohomology of $C_X$ is given by
$$\widetilde H^*(C_X) = H^*(X \times X, X) = \ker (H^*(X \times X) \stackrel{\Delta^*}\to H^*(X)),$$ which is the ideal of zero-divisors in $H^*(X \times X)$. The result follows.
\end{proof}

Since the condition $H(\alpha)\neq0$ can hold only with $q\geq2p-1$ (and $p\geq1$), the only instances of Theorem~\ref{t1} with $H(\alpha)\neq0$ that have not yet been proved are those with $q=2p-1$ and $p\geq1$.

\begin{proof}[Proof of Theorem~\ref{t1}] (Assuming $H(\alpha)\neq0$, $q=2p-1$, and $p\geq2$.) It has been shown in~\cite[Theorem~5.2]{GGVhopf} that $\TC(X)=\zcl_{\mathbb{Z}}(X)=4$. Further $\TC(X)\geq\cat(C_X)$ in view of~\cite[Theorem~10]{GCVcofibre}. The result then follows from Lemma~\ref{markweigths}.
\end{proof}

\begin{proof}[Proof of Theorem~\ref{t1}] (Assuming $H(\alpha)\neq0$ and $p=q=1$, where $H(\alpha)$ is to be interpreted as $\deg(\alpha)$.) The previous argument works (using $\mathbb{Z}_2$ coefficients) when $\deg(\alpha)=\pm2$, whereas the situation is elementary for $\deg(\alpha)=\pm1$. Lastly, as detailed below, the argument in~\cite[Theorem~5.1]{GGVhopf} proving
\begin{equation}\label{reciclado}
\mbox{$\TC(X)=4\;$ for $\;|\deg(\alpha)|>2$}
\end{equation}
is easily extended to show $\TC(X)=\cat(C_X)=4$.

\smallskip
Let $k$ stand for the absolute value of $\deg(\alpha)$, and consider  generators $x_i$ of $H^i(X;\mathbb{Z}_k)=\mathbb{Z}_k$, for $i=1,2$, connected by the mod-$k$ Bockstein operator $\beta_k$. Then the corresponding zero-divisors $\bar{x}_i=1\times x_i-x_i\times1\in H^i(X\times X;\mathbb{Z}_k)$ are connected by $\beta_k$ too. As observed in the proof of Lemma~\ref{markweigths}, the latter cohomology classes can be thought of as lying in $H^*(C_X;\mathbb{Z}_k)$, where they have to be connected by $\beta_k$. Then~\cite[Theorem 3.12]{MR1317569} implies that the class $\bar{x}_2\in H^2(C_X;\mathbb{Z}_k)$ has category weight at least 2 and, since the square of the latter class is obviously non-zero (recall $k>2$), we obtain $\cat(C_X)\geq4$. The result now follows from~(\ref{reciclado}) and~\cite[Theorem~10]{GCVcofibre}.
\end{proof}

\begin{proof}[Proof of Theorem~\ref{t1}] (Assuming $2\leq p\leq q\leq3(p-1)$ and $H(\alpha)=0$.)
It is well known that $\TC(X)=\zcl_{\mathbb{Q}}(X)=2$ (see~\cite{MR3117387} and the initial considerations in Section~5 of~\cite{GGVhopf}). The result follows again from~\cite[Theorem~10]{GCVcofibre} and~Lemma~\ref{markweigths}.
\end{proof}

\begin{proof}[Proof of Theorem~\ref{t1}] (Assuming $p=q=1$ and $H(\alpha)=0$.)
 Here $X=S^1\vee S^2$, $\TC(X)=\zcl(X)=2$, and $\cat(C_X)\geq2$. Since $X$ is a suspension, Remark \ref{cofibresuspension} gives $\cat(C_X)\leq 2$ and completes the proof.
\end{proof}

\begin{rem} In fact, by combining Remark \ref{cofibresuspension} and Lemma \ref{markweigths} with the methods and results of \cite{MR3117387}, it is not difficult to show that $\TC(X)=\cat(C_X)$ whenever $X$ is a path-connected suspension of finite type.
\end{rem}

\bibliographystyle{plain}
\bibliography{bib}

\end{document}